\documentclass[a4paper,12pt]{article}

\usepackage[utf8]{inputenc}
\usepackage[T1]{fontenc}
\usepackage[left=2cm,top=2.5cm,right=2cm,bottom=2.5cm]{geometry}
\usepackage{amsthm}
\usepackage{amsmath,amssymb}

\usepackage{url}

\usepackage{enumitem}

\newtheorem{teo}{Theorem}[section] 
\newtheorem*{teo*}{Theorem}
\newtheorem{lem}[teo]{Lemma} 
\newtheorem{prop}[teo]{Proposition}

\newtheorem{definicion}{Definition}[section]

\newcommand{\mc}{\mathcal}

\newcommand{\R}{\mathbb{R}}\newcommand{\Rn}{\R^n}\newcommand{\Rnn}{\R^{n\times n}}
\newcommand{\N}{\mathbb{N}}

\renewcommand{\O}{\Omega}

\newcommand{\g}{\gamma}

\renewcommand{\d}{\delta}

\renewcommand{\t}{\theta}

\newcommand{\p}{\partial}
\newcommand{\f}{\varphi}

\newcommand{\weakc}{\rightharpoonup}

\DeclareMathOperator{\dist}{dist}

\DeclareMathOperator{\tr}{tr}

\renewcommand{\div}{\operatorname{div}}
\DeclareMathOperator{\cof}{cof}

\DeclareMathOperator{\supp}{supp}

\DeclareMathOperator{\diver}{div}

\usepackage{amsmath}
\usepackage{latexsym}
\usepackage{amssymb}

\usepackage{color}

\usepackage{appendix} 

\title{Minimizers of Nonlocal Polyconvex Energies in Nonlocal Hyperelasticity}
\author{Jos\'e C. Bellido$^a$, Javier Cueto$^b$ and Carlos Mora-Corral$^c$ \vspace*{0.25cm}\\
	{\footnotesize $^a$ E.T.S.I.\ Industriales, Department of Mathematics, Universidad de Castilla-La Mancha,} \\
	{\footnotesize 13071-Ciudad Real, Spain. Email \url{JoseCarlos.Bellido@uclm.es}},\\
	{\footnotesize $^b$ Department of Mathematics, University of Nebraska-Lincoln,} \\
	{\footnotesize Lincoln NE 68588-0130, US. Email \url{jcuetogarcia2@unl.edu}}, \\
	
	{\footnotesize $^c$ Departamento de Matem\'aticas, Universidad Aut\'onoma de Madrid,} \\
	{\footnotesize 28049 Madrid, Spain  and Instituto de Ciencias Matem\'aticas,} \\
	{\footnotesize CSIC-UAM-UC3M-UCM, 28049 Madrid, Spain. Email \url{Carlos.Mora@uam.es}}}

\date{}

\begin{document}

\maketitle

\pagestyle{empty}

\begin{abstract}
We develop a theory of existence of minimizers of energy functionals in vectorial problems based on a nonlocal gradient under Dirichlet boundary conditions.
The model shares many features with the peridynamics model and is also applicable to nonlocal solid mechanics, especially nonlinear elasticity.
This nonlocal gradient was introduced in an earlier work, inspired by Riesz' fractional gradient, but suitable for bounded domains.
The main assumption on the integrand of the energy is polyconvexity.
Thus, we adapt the corresponding results of the classical case to this nonlocal context, notably, Piola's identity, the integration by parts of the determinant and the weak continuity of the determinant.
The proof exploits the fact that every nonlocal gradient is a classical gradient. 
\end{abstract}





\section{Introduction}

In nonlinear elasticity, and, specifically, in the static theory, a fundamental problem is the existence of minimizers of the elastic energy
\begin{equation}\label{eq:elastic}
 \int_{\O} W (x, u(x), D u (x)) \, dx
\end{equation}
of a deformation $u : \O \to \Rn$.
Here $\O$ is an open bounded subset of $\Rn$ representing the reference configuration of the body (of course, only $n=3$ is  physicallly relevant), and $W : \O \times \Rnn \to \R \cup \{\infty\}$ is the elastic stored-energy function of the material.
In fact, the dependence on $u(x)$ in \eqref{eq:elastic} is not significant.
The usual approach for finding such minimizers is the direct method of the Calculus of variations.
As shown in the pioneering paper of Ball \cite{Ball77}, the weak continuity of the determinant of the deformation gradient $D u$, and the assumption of polyconvexity of $W$ are the key ingredients to make the direct method work.
The property of polyconvexity essentially means that $W$ can be expressed as a convex function of the minors of the matrix $Du$.

This theory is by now well established.
On the other hand, nonlocal models in solid mechanics have experienced a huge development in the last two decades, especially since the introduction of the peridynamics model by Silling \cite{Silling2000}. This prominence is due among other things to the fact that their formulation usually allows for less regular functions, in particular functions exhibiting singularity phenomena.
Many refinements have been introduced since then and, particularly, nonlocal models based on a nonlocal gradient have received a great attention as an adequate substitute of local models.
In general, a nonlocal gradient of a function $u : \O \to \R$ takes the form 
\[
 \mathcal{G}_\rho u(x)= \int_\O \frac{u(x)-u(y)}{|x-y|} \frac{x-y}{|x-y|} \rho(x-y)\, dy ,
\]
for a suitable kernel $\rho$, usually with a singularity at the origin.
The choice of $\rho$ determines the nonlocal gradient, which, in turn, specifies the functional space.

The most popular nonlocal gradient is possibly Riesz' $s$-fractional gradient, which is denoted by $D^s u$ and corresponds to the choices $\O = \Rn$ and $\rho(x) = \frac{c_{n,s}}{|x|^{n-1+s}}$ for some constant $c_{n,s}$; see \cite{ShS2015,ShS2018}.
Here $0<s<1$ is the degree of differentiability.
While Riesz' fractional gradient enjoys many desirable properties (it is invariant under rotations and translations, it is $s$-homogeneous under dilations; see \cite{Silhavy2019}), it has the drawback that the integral defining it is over the whole space, which makes it unsuitable for solid mechanics where the body is represented by a bounded domain $\O \subset \Rn$.
An adaptation of Riesz' $s$-fractional gradient for bounded domains was done by the authors in \cite{BeCuMo22}.
Precisely, for a $C^{\infty}_c$ function $u$, its nonlocal gradient is defined as
\[
 D_\delta^s u(x) = c_{n,s} \int_{B(x,\delta)} \frac{u(x)-u(y)}{|x-y|}\frac{x-y}{|x-y|}\frac{w_\d(x-y)}{|x-y|^{n-1+s}} \, dy ,
\]
where $w_{\d}$ is a fixed function in $C^{\infty}_c (B(0, \d))$ satisfying some natural properties to be a truly cut-off function. The parameter $\d>0$, called \emph{horizon} in the context of peridynamics, indicates the maximum interaction distance.

Although the definition of the Riesz gradient is rather old, it was the study \cite{ShS2015,ShS2018} that initiated the attention in the community of nonlocal problems in partial differential equations and Calculus of variations.
In fact, they showed that the functional space consisting of the closure of smooth functions under the natural norm (given by the $L^p$ norms of a function $u$ and its fractional gradient $D^s u$) is the Bessel potential space $H^{s,p} (\Rn)$. They also proved a series of embeddings mimicking those of Sobolev spaces.
As a consequence of their analysis, they proved the existence of minimizers for functionals of the form
\begin{equation}\label{eq:integralDs}
 \int_{\Rn} W (x, u(x), D^s u(x)) \, dx ,
\end{equation}
where $u : \Rn \to \R$ is scalar and $W$ is convex in the last variable.
The vectorial case, for $u : \Rn \to \R^m$, was treated in \cite{BeCuMC} under the assumption of polyconvexity of $W$, and in \cite{KrSc22} under the assumption of quasiconvexity.

In \cite{BeCuMo22} we developed a framework for the nonlocal gradient $D^s_{\d} u$ over bounded domains and its associated function space $H^{s,p,\d} (\O)$, which roughly can be defined as the set of $L^p$ functions $u$ with $L^p$ gradients $D^s_{\d} u$. In this paper we use that analysis and, additionally, take advantage of the large body of knowlegde on hyperelasticity based on polyconvexity developed in the last four decades for functionals of the form \eqref{eq:elastic}.
With these two ingredients we show the existence of minimizers of functionals
\begin{equation}\label{eq:integralDsd}
 \int_{\O} W (x, u(x), D^s_{\d} u(x)) \, dx ,
\end{equation}
for $u : \Rn \to \Rn$ under the assumption of polyconvexity of $W$. This article can be seen as the goal of the path shown in the thesis \cite{Cue21}, following the articles \cite{BeCuMCBondBased,BeCuMC,bellido2020convergence,BeCuMo22}: we first showed that the nonlinear bond-based description of peridynamics does not fit in solid mechanics, then we studied the energy functionals \eqref{eq:integralDs} based on the Riesz fractional gradient as an alternative (instead of the double integrals present in peridynamics), and finally we presented this similar framework over bounded domains (see \eqref{eq:integralDsd}).

In order to obtain the theory of polyconvexity in this framework, one can follow two approaches.
One is to obtain directly the necessary tools and follow the steps from the proofs of the local case, as was done by the authors in \cite{BeCuMC}, where we developed the theory of polyconvexity for Riesz' fractional gradient.
This procedure requires the validity of Piola's identity and the integration by parts of the determinant as preliminary steps to prove the weak continuity of minors, all in the fractional context.

The other approach starts from the observation that every nonlocal gradient is in fact a local one.
This method was exploited in \cite{KrSc22} for Riesz' fractional gradient to develop the theory of quasiconvexity in that context. In the fractional case, every fractional gradient coincides with the classical one of a locally Sobolev function: $D^s u = D(I_s * u)$, being $I_s$ the Riesz potential, which is $I_s (x) = \frac{1}{\gamma(s)|x|^{n-s}}$, for some constant $\gamma(s)$.
However, the fact that the property of compact support is not preserved under convolution, as well as the non-integrability of the Riesz potential makes it difficult to use this observation for obtaining fundamental results in the polyconvexity theory based on Riesz gradients.
On the contrary, these two disadvantages are not present in the context of this paper, and that is why we adopt this approach here.
Indeed, the starting point is the equality $D^s_{\d} u = D (Q^s_{\d} * u)$ for a suitable kernel $Q^s_{\d}$, which is integrable and of compact support.
Therefore, this second method is fully operative for the polyconvexity theory in the context of the nonlocal gradient $D^s_{\d}$.
In fact, we have found two advantages: first, that it is simpler and provides shorter proofs, whenever available.
Second, and more importantly, that it gives rise to better integrability exponents, which results in weaker assumptions in the existence theorem.
This second reason is supported by the sharp results that are available for the weak continuity of the determinant in the classical case, notably due to \cite{MuQiYa94}. This sharpness of the exponents actually comes from the integration by parts of the determinant, since here it is written completely in terms of $Q_\d^s*u$, as opposed to the fractional case.

In a nutshell, the existence theory of polyconvexity is based on the weak continuity of the determinant, which in turn depends on the integration by parts of the determinant, which, in fact is founded on Piola's identity.
In order to develop our existence theory we could just start from the weak continuity of the determinant, but we have preferred to prove Piola's identity and the integration by parts of the determinant as well, for the sake of completeness and also for comparison with the local case.

The structure of this paper is as follows.
In Section \ref{se:functional} we recall from \cite{BeCuMo22} the main results about the nonlocal gradient $D^s_{\d} u$, as well as its associated functional space $H^{s,p,\d} (\O)$.
In Section \ref{se:preliminaries} we prove some properties of the nonlocal gradient and divergence, notably, the product formula.
In Section \ref{se:from} we comment on the fact that every nonlocal gradient is, in fact, a local one.
This property is exploited in Section \ref{se:Piola}, where we prove the nonlocal versions of Piola's identity, the integration by parts of $\det D^s_{\d} u$ and the weak continuity of $\det D^s_{\d} u$.
Finally, in Section \ref{se:existence} we prove the existence of minimizers of functionals \eqref{eq:integralDsd} with $W$ polyconvex.
We also show the corresponding Euler--Lagrange equation.

\section{Nonlocal gradient, divergence and associated space}\label{se:functional}

This section is a compendium of the definitions and results taken from \cite{BeCuMo22} on the nonlocal gradient and divergence, as well as their associated space $H^{s,p,\d} (\O)$.

Throughout the article, $\O$ is a non-empty bounded open set of $\Rn$.
We fix $0 < s < 1$ (the degree of differentiability) and $\d>0$ (the horizon distance).
We also define the sets $\O_{\d}= \O + B(0,\d)$, which plays the role of nonlocal closure, and $\O_{B, \d} = \O_{\d} \setminus \O$, which plays the role of nonlocal boundary.
We write $B(x,r)$ to denote the open ball centred at $x$ of radius $r$.
The set $\O_{-\d} := \{ x \in \O : \dist (x, \p \O) > \d \}$ will also be relevant, where $\d$ is chosen small enough so that $\O_{-\d}$ is not empty.

Let $w_\d: \Rn\to [0,+\infty)$ be a cut-off function, and  $\rho_\d: \Rn\to [0,+\infty)$ defined as 
\[\rho_\delta(x)=\frac{1}{\gamma (1-s)|x|^{n-1+s}}w_{\d}(x),\] 
with $0<s<1$, where the constant $\gamma(s)$ is given by
\begin{equation*}
	\gamma(s)=\frac{\pi^{\frac{n}{2}}2^s\Gamma\left(\frac{s}{2}\right)}{\Gamma\left(\frac{n-s}{2}\right)}
\end{equation*}
and $\Gamma$ is Euler's gamma function.
The precise assumptions over $w_{\d}$ are as follows:
\begin{enumerate}[label=\alph*)]
	\item $w_\d$ is radial and nonnegative.
	\item $w_\d \in C_c^\infty(B(0,\delta))$.
	\item There are constants $a_0>0$ and $0<b_0<1$ such that $0\leq w_\d \leq a_0$, and $w_\d |_{B(0, b_0\delta)} =a_0$.
	\item
	$w_\d (x_1) \geq w_\d (x_2)$ if $|x_1| \leq |x_2|$.
\end{enumerate}
Note that $\rho_\d \in L^1 (\Rn)$.

The definitions of the nonlocal gradient and divergence for smooth functions are the following.
\begin{definicion} \label{def: nonlocal gradient}
Set
	\begin{equation*}
		c_{n,s}:= \frac{n-1+s}{\gamma(1-s)}.
	\end{equation*}
	\begin{enumerate}[label=\alph*)]
		\item \label{item:Dsdu} Let $u\in C_c^{\infty} (\Rn)$. The nonlocal gradient $D_\delta^s u$ is defined as
		\begin{equation*}
			D_\delta^s u(x)= c_{n,s} \int_{B(x,\delta)} \frac{u(x)-u(y)}{|x-y|}\frac{x-y}{|x-y|}\frac{w_\d(x-y)}{|x-y|^{n-1+s}} \, dy , \qquad x \in \Rn.
		\end{equation*}
		\item
		 Let $u \in C^1_c (\Rn,\Rn)$. The nonlocal divergence is defined as
		\begin{equation*}
			\diver_{\delta}^s u(x)=  c_{n,s} \int_{B(x,\delta)} \frac{u(x)-u(y)}{|x-y|}\cdot\frac{x-y}{|x-y|}\frac{w_\d(x-y)}{|x-y|^{n-1+s}} \, dy, \qquad x \in \Rn ,
		\end{equation*}
and the nonlocal gradient is defined as
\begin{equation*}
 D_\delta^s u(x)= c_{n,s} \int_{B(x,\delta)} \frac{u(x)-u(y)}{|x-y|} \otimes \frac{x-y}{|x-y|}\frac{w_\d(x-y)}{|x-y|^{n-1+s}} \, dy , \qquad x \in \Rn.
\end{equation*}
	\end{enumerate}
\end{definicion}
Notice that the three integrals in Definition \ref{def: nonlocal gradient} are absolutely convergent because $u$ is Lipschitz and $\rho_{\d} \in L^1 (\Rn)$.
It is also immediate from the definition that $\supp D_\delta^s u \subset \supp u + B (0, \d)$, where $\supp$ denotes the support of a function.


The operators of Definition \ref{def: nonlocal gradient} are dual operators in the sense of integration by parts.
Several versions of integration by parts formulas for related fractional or nonlocal operators have appeared in the literature \cite{DGLZ,MeS,COMI2019,Silhavy2019}.
The integration by parts formula of interest in this investigation is the following \cite[Th.\ 3.2]{BeCuMo22}.

\begin{teo}
Suppose that $u \in C^{\infty}_c (\Rn)$ and $\phi \in C^1_c (\O, \Rn)$.
Then $D^s_{\d} u \in L^{\infty} (\Rn, \Rn)$ and $\diver^s_{\d} \phi \in L^{\infty} (\Rn)$.
Moreover,
	\begin{align*}
		\int_{\O} D_\delta^s u(x) \cdot \phi(x) \, dx =& - \int_{\O} u(x) \diver_\delta^s \phi (x) \, dx \\
		&- (n-1+s) \int_{\O} \int_{\O_{B, \d}} \frac{u(y)\phi(x)}{|x-y|} \cdot \frac{x-y}{|x-y|}\rho_\delta(x-y) \, dy \, dx
	\end{align*}
and these three integrals are absolutely convergent.
\end{teo}

We now extend Definition \ref{def: nonlocal gradient}\,\ref{item:Dsdu} to a broader class of functions.

\begin{definicion}\label{def: nonlocal gradient 2}
	\begin{enumerate}[label=\alph*)]
		\item
		Let $u \in L^1 (\O_\d)$ be such that there exists a sequence of $\{ u_j \}_{j \in \N} \subset C^{\infty}_c (\Rn)$ converging to $u$ in $L^1 (\O_\d)$ and for which $\{ D_\d^s u_j \}_{j \in \N}$ converges to some $U$ in $L^1 (\O, \Rn)$.
		We define $D_\d^s u$ as $U$.
		
		\item
		Let $\phi \in L^1 (\O_\d, \Rn)$ be such that there exists a sequence of $\{ \phi_j \}_{j \in \N} \subset C^{\infty}_c (\Rn, \Rn)$ converging to $\phi$ in $L^1 (\O_\d, \Rn)$ and for which $\{ \diver_\d^s \phi_j \}_{j \in \N}$ converges to some $\Phi$ in $L^1 (\O)$.
		We define $\diver_\d^s \phi$ as $\Phi$.
	\end{enumerate}
\end{definicion}

It was shown in \cite[Lemma 3.3]{{BeCuMo22}}
that the above definitions of $D_\d^s u$ and $\diver_\d^s \phi$ are independent of the sequence chosen.


%

Similarly to the definition of Bessel spaces $H^{s,p}(\Rn)$ \cite{BeCuMC,COMI2019}, we define the space object of our study. 

\begin{definicion}\label{de:Hspd}
Let $1 \leq p < \infty$.
We define the space $H^{s,p,\d}(\O)$ as the closure of $C^{\infty}_c(\Rn)$ under the norm 
\[
 \left\| u \right\| _{H^{s,p,\d}(\O)} = \left( \left\| u \right\|_{L^p (\O_{\d})}^p + \left\| D_\delta^s u \right\|_{L^p (\O)}^p \right)^{\frac{1}{p}} .
\]
\end{definicion}
Thus, functions in $H^{s,p,\d}(\O)$ are defined a.e.\ in $\O_{\d}$, while its gradient (Definition \ref{def: nonlocal gradient 2}) is defined a.e.\ in $\O$.
Definition \ref{de:Hspd} was introduced in \cite{BeCuMo22}.
There is an alternative definition of these spaces, based on a distributional notion of $D^s_{\d} u$; accordingly, $H^{s,p,\d}(\O)$ is also the set of functions $u \in L^p (\O_{\d})$ whose distributional $D^s_{\d} u$ is in $L^p (\O)$; see \cite{CuKrSc22}, where one can find the equivalence with the current one.
The situation is, thus, similar to what happens with the Bessel potential spaces $H^{s,p} (\Rn)$: they can be defined as the closure of smooth functions \cite{ShS2015} and as the set of functions $u \in L^p (\Rn)$ whose distributional gradient $D^s u$ is in $L^p (\Rn)$ \cite{COMI2019}.

The case $p=\infty$ is avoided in Definition \ref{de:Hspd}.
Nevertheless, one can give the following definition based on a type of closure: $u \in H^{s,\infty,\d}(\O)$ when there exists a sequence $\{ u_j \}_{j \in \N}$ in $C^{\infty}_c (\Rn)$ such that $u_j \to u$ uniformly in $\O_{\d}$, $D^s_{\d} u_j \to D^s_{\d} u$ a.e.\ in $\O$ and $\| D^s_{\d} u_j \|_{L^{\infty} (\O)} \to \| D^s_{\d} u \|_{L^{\infty} (\O)}$.
This definition is inspired by the density properties of Sobolev functions in the case $p=\infty$; see, e.g., \cite[Exercise 10.21]{Leoni09}.
It is not difficult to show that all theorems stated here for $p< \infty$ can be proved with slight adaptations to the case $p=\infty$.


The space $H^{s,p,\d}(\O)$ satisfies reflexivity and separability properties.  
See \cite[Prop 3.4]{BeCuMo22}.
\begin{prop} \label{prop: espacio separable y reflexivo}
Let $1\leq p < \infty$.
Then $H^{s,p,\d}(\Omega)$ is a separable Banach space.
If, in addition, $p > 1$, it is reflexive.
\end{prop}

The following continuous inclusion holds immediately.
\begin{prop}\label{pr:inclusionH}
The continuous inclusion $H^{s,p,\d}(\O) \subset H^{s,q,\d}(\O)$ holds whenever $1\leq q \leq p < \infty$.
\end{prop}

In order to describe the boundary condition, we recall the set $\O_{-\d} = \{ x \in \O : \dist (x, \p \O) > \d \}$ and define the subspace $H_0^{s,p,\d}(\O_{-\d})$ as the closure of $C_c^{\infty}(\O_{-\d})$ in $H^{s,p,\d}(\O)$.
It is immediate to check that any $u \in H_0^{s,p,\d}(\O_{-\d})$ satisfies $u=0$ a.e.\ in $\O_{\d} \setminus \O_{-\d}$.
Finally, given $g \in H^{s,p,\d}(\O)$ we define the affine subspace $H^{s,p,\d}_g (\O_{-\d})$ as $g+ H^{s,p,\d}_0(\O_{-\d})$.


An essential tool for obtaining existence of minimizers for integral functionals is a Poincar\'e-type inequality. 
Given $p > 1$ and $0<s<1$ with $sp<n$ we define $p_s^* := \frac{np}{n-sp}$.

\begin{teo}\label{th:PoincareSobolev delta}
Let $1<p<\infty$.
Then there exists $C=C(|\O|,n,p,s)>0$ such that
\begin{displaymath}
\lVert u \rVert_{L^q (\O)} \leq C \lVert D_\d^s u \rVert_{L^p(\O)}
\end{displaymath}
for all $u \in H_0^{s,p,\d}(\O_{-\d})$, and any $q$ satisfying
\[
\begin{cases}
 q\in \left[1, p_s^* \right] & \text{if } sp<n , \\
 q\in[1, \infty) & \text{if } sp=n , \\
 q\in[1,\infty] & \text{if } sp>n .
\end{cases}
\]
\end{teo}
\begin{proof}
The case $sp<n$ corresponds to \cite[Th.\ 6.1]{BeCuMo22}.
The case $sp=n$ is a consequence of \cite[Th.\ 6.4]{BeCuMo22} (or else of the previous case and Proposition \ref{pr:inclusionH}).
The case $sp>n$ is a particular case of \cite[Th.\ 6.3]{BeCuMo22}.
\end{proof}

The following result 
decides which of the embeddings of Theorem \ref{th:PoincareSobolev delta} are compact.
We will indicate by $\weakc$ weak convergence.

\begin{teo}\label{Hspdelta embedding theorem}
Let $1<p<\infty$ and $g \in H^{s,p,\d}(\O)$. Then for any sequence $\{u_j \}_{j \in \N} \subset  H_g^{s,p,\d}(\Omega_{-\d})$
such that
\begin{equation*}
u_j \rightharpoonup u \quad \text{in } H^{s,p,\d} (\O),
\end{equation*}
for some $u \in H^{s,p,\d} (\O)$, one has $u \in H^{s,p,\d}_g (\O)$ and
\begin{equation*}
u_j \rightarrow u \quad \text{in } L^q (\O),
\end{equation*}
for every $q$ satisfying
\[
\begin{cases}
 q\in \left[1, p_s^* \right) & \text{if } sp<n , \\
 q\in[1, \infty) & \text{if } sp = n , \\
 q\in[1, \infty] & \text{if } sp > n .
\end{cases}
\] 
\end{teo}

The case $p=1$ is not covered by Theorem \ref{th:PoincareSobolev delta}, since the result \cite[Th.\ 6.1]{BeCuMo22} was not able to deal with that case.
On the contrary, there is a version of Theorem \ref{Hspdelta embedding theorem} for $p=1$ in \cite[Th.\ 7.3]{BeCuMo22}, but, for simplicity, we have not mentioned it since it will not be used, as the existence theory of Section \ref{se:existence} deals with the reflexive case.

\section{Calculus in $H^{s,p,\d}$}\label{se:preliminaries}

In this section we present a product formula for the nonlocal derivative and divergence.

We start with the embeddings from Sobolev spaces $W^{1,p}$ to $H^{s,p,\d}$, as well as an inequality between the norms of the classical and the nonlocal gradient.
Although this result is not actually needed in the sequel, we have included it in order to locate the regularity of $H^{s,p,\d}$ in comparison with that of $W^{1,p}$. At the end of this article we will comment on functions in the latter functional spaces that are not in the corresponding Sobolev ones.

\begin{lem} \label{Lemma nonlocal gradient bound} 
	Assume that $\O_{\d}$ has a Lipschitz boundary. Then:
\begin{enumerate}[label=\alph*)]
\item\label{item:WHa} Let $1\leq p < \infty$.
The continuous embedding $W^{1,p}(\O_\d) \subset H^{s,p,\d} (\O)$ holds.
Moreover, for all $u \in W^{1,p}(\O_\d)$,
	\begin{equation}\label{eq:Sobolevbound}
 	\|D_\d^s u\|_{L^p(\O)} \leq (n-1+s) \left\| \rho_{\d} \right\|_{L^1 (\Rn)} \left\| Du \right\|_{L^p(\O_{\d})} .
	\end{equation}	

\item\label{item:WHb} 
The continuous embedding $W^{1,\infty}(\O_\d) \subset H^{s,p,\d} (\O)$ holds for all $p \in [1, \infty)$.
Moreover, for all $u \in W^{1,\infty}(\O_\d)$,
	\begin{equation}\label{eq:Sobolevboundinf}
 	\|D_\d^s u\|_{L^{\infty}(\O)} \leq (n-1+s) \left\| \rho_{\d} \right\|_{L^1 (\Rn)} \| D u\|_{L^{\infty}(\O_\d)} .
	\end{equation}	
\end{enumerate}
\end{lem}
\begin{proof}
We start with \ref{item:WHa}.
We first assume $u \in C_c^{\infty}(\Rn)$.
Applying Minkowski's integral inequality (see, e.g., \cite[App.\ A.1]{Stein70}) to the $L^p$ norm of $D^s_{\d} u$, we have
	\begin{equation} \label{eq: sobolev embedding 1}
\begin{split}
		&\left(\int_{\O} \left|\int_{B(x,\d)} \frac{u(x)-u(y)}{|x-y|} \frac{x-y}{|x-y|} \frac{w_\d(x-y)}{|x-y|^{n+s-1}} \, dy\right|^p dx\right)^{\frac{1}{p}} \leq \\
	&\int_{B(0,\d)} \left( \int_{\O}\left(\frac{|u(x)-u(x-h)|}{|h|^{n+s}}w_\d(h)\right)^p dx \right)^{\frac{1}{p}} dh .
\end{split}
	\end{equation}
	Now, for all $h\in B(0,\d) \setminus \{0 \}$,
	\begin{equation} \label{eq: sobolev embedding 2}
	\begin{split}
		\left(\int_{\O}\left(\frac{|u(x)-u(x-h)|}{|h|^{n+s}}w_\d(h)\right)^p dx\right)^{\frac{1}{p}} =& \g (1-s) \rho_{\d} (h) \left(\int_{\O} \left| \frac{u(x)- u(x-h)}{|h|} \right|^p dx\right)^{\frac{1}{p}} \\
		&\leq \g (1-s) \rho_{\d} (h) \| D u \|_{L^p(\O_\d)} ,
	\end{split}
	\end{equation}
	where we have used a classic inequality on the $L^p$ estimate of translations \cite[Prop.\ 9.3]{Brezis}.
Combining \eqref{eq: sobolev embedding 1} and \eqref{eq: sobolev embedding 2} we obtain inequality \eqref{eq:Sobolevbound}.

Now we show the inequality for any $u \in W^{1,p}(\O_\d)$ through an extension and density argument.
Let $\tilde{u} \in W^{1,p}(\Rn)$ be an extension of $u$, and let $\{ u_j \}_{j \in \N}$ be a sequence in $C^{\infty}_c (\Rn)$ converging to $\tilde{u}$ in $W^{1,p}(\Rn)$.
Then $\{ u_j \}_{j \in \N}$ also converges to $u$ in $H^{s,p,\d}(\O)$; indeed, this can be shown by the string of embeddings
\[
 W^{1,p}(\Rn) \subset H^{s,p} (\Rn) \subset H^{s,p, \d} (\O) ;
\]
see \cite[Ch.\ 1]{Adams} or \cite[Prop.\ 2.7]{bellido2020convergence} for the first inclusion, and \cite[Prop.\ 3.5]{BeCuMo22} for the second.
By inequality \eqref{eq:Sobolevbound},
\[
 \| D_\d^s u_j \|_{L^p(\O)} \leq(n-1+s) \left\| \rho_{\d} \right\|_{L^1 (\Rn)} \| D u_j \|_{L^p(\O_\d)} , 
\]
and passing to the limit as $j \to \infty$, we obtain \ref{item:WHa}.

For the proof of \ref{item:WHb} we rely on \ref{item:WHa}.
Let $u \in W^{1,\infty}(\O_\d)$.
Then $u \in W^{1,p}(\O_\d)$ for all $p \in[1, \infty)$, so $u \in H^{s,p,\d} (\O)$ and
\[
 \|D_\d^s u\|_{L^p(\O)} \leq(n-1+s) \left\| \rho_{\d} \right\|_{L^1 (\Rn)} \left\| Du \right\|_{L^p(\O_{\d})} \leq (n-1+s) \left\| \rho_{\d} \right\|_{L^1 (\Rn)} |\O|^{\frac{1}{p}} \| D u\|_{L^{\infty}(\O_\d)}.
\]
Letting $p \to \infty$ we obtain inequality \eqref{eq:Sobolevboundinf}, so \ref{item:WHb} is proved.
\end{proof}

Note that one cannot replace $\| D u\|_{L^p(\O_\d)}$ with $\| D u\|_{L^p(\O)}$ in inequality \eqref{eq:Sobolevbound}, since, if $\| D u\|_{L^p(\O)} = 0$ then $u$ is constant in $\O$, hence $D_\d^s u = 0$ in $\O_{-\d}$, but $D_\d^s u$ is not necessarily zero in $\O \setminus \O_{-\d}$.

Now we introduce a nonlocal operator similar to the nonlocal gradient and divergence that plays an essential part in the derivative of a product. The fractional analogue of this operator was studied in \cite{BeCuMC}.
In truth, the same symbol $K^{s,\d}_{\f}$ denotes four slightly different operators, but the notation chosen avoids any risk of confusion.
Henceforth, $[\varphi]_{C^{0,1}(\O_\d)}$ denotes the Lipschitz seminorm of $\varphi$ in $\O_\d$.

\begin{definicion}\label{de:Ksd}
Let $\f \in  C^{0,1}(\O_\d)$.
\begin{enumerate}[label=\alph*)]
\item For $U \in L^1 (\O_{\d})$ or $U \in L^1 (\O_{\d}, \Rnn)$ we define
\[
 K^{s,\d}_{\varphi}(U)(x)= c_{n,s} \int_{B(x,\d)} \frac{\varphi(x)-\varphi(y)}{|x-y|^{n+s}} U(y) \frac{x-y}{|x-y|}w_\d(x-y) \, dy , \qquad \text{a.e. } x \in \O .
\]

\item For $U \in L^1 (\O_{\d}, \Rn)$ we define
\[
 K^{s,\d}_{\varphi}(U)(x)= c_{n,s} \int_{B(x,\d)} \frac{\varphi(x)-\varphi(y)}{|x-y|^{n+s}} U(y) \cdot \frac{x-y}{|x-y|}w_\d(x-y) \, dy , \qquad \text{a.e. } x \in \O
\]
and
\[
 K^{s,\d}_{\varphi}(U^T)(x)= c_{n,s} \int_{B(x,\d)} \frac{\varphi(x)-\varphi(y)}{|x-y|^{n+s}} U(y) \otimes \frac{x-y}{|x-y|}w_\d(x-y) \, dy , \qquad \text{a.e. } x \in \O .
\]
\end{enumerate}
\end{definicion}

The next result is the analogue of \cite[Lemma 3.2]{BeCuMC} in a bounded domain framework.

\begin{lem} \label{Lema operador lineal delta} 
Let $1 \leq p \leq \infty$.
Let $\varphi \in  C^{0,1}(\O_\d)$.
Then:
\begin{enumerate}[label=\alph*)]
\item\label{item:Kbounded} The operators
\begin{align*}
 K^{s,\d}_{\f}: L^p(\O_\d) & \to L^p(\O, \Rn) , & K^{s,\d}_{\f}: L^p(\O_\d, \Rnn) & \to L^p(\O, \Rn)
\end{align*}
and
\begin{align*}
 K^{s,\d}_{\f}: L^p(\O_\d, \Rn) & \to L^p(\O), & K^{s,\d}_{\f}: L^p(\O_\d, \Rn) & \to L^p(\O, \Rnn) \\
 U & \mapsto K^{s,\d}_{\f} (U) , & U & \mapsto K^{s,\d}_{\f} (U^T)
\end{align*}
are linear and bounded, and in all cases we have the estimate
	\begin{equation*}
	\|K_{\varphi}^{s,\d} (U) \|_{L^p (\O)} \leq (n+s-1) [\f]_{0,1} \left\| \rho_{\d} \right\|_{L^1 (\Rn)} \| U \|_{L^p (\O_{\d})} 
	\end{equation*}
and analogously for $\|K_{\varphi}^{s,\d} (U^T) \|_{L^p (\O)}$.

\item\label{item:trK} $\tr K^{s,\d}_{\varphi}(U^T) = K^{s,\d}_{\varphi}(U)$ for every $U \in L^1 (\O_{\d}, \Rn)$.

\end{enumerate}
\end{lem}

\begin{proof}
For \ref{item:Kbounded} we will do the case for the operator $K^{s,\d}_{\f}: L^p(\O_\d) \to L^p(\O, \Rn)$, as the same proof is valid for all four operators, which are clearly linear.
So let $U \in L^p (\O_\d)$.
The steps are similar to the proof of Lemma \ref{Lemma nonlocal gradient bound}.
Assume first $p < \infty$.
By Minkowski's integral inequality
\[
\begin{split}
	 &\left( \int_{\O} \left| \int_{B(x,\d)} \frac{\varphi(x)-\varphi(y)}{|x-y|^{n+s}} U(y) \frac{x-y}{|x-y|}w_\d(x-y) \, dy \right|^p \, dx \right)^{\frac{1}{p}}\leq \\
	 &\int_{B(0,\d)} \left( \int_{\O} \left|  \frac{\f(x) - \f(x-h)}{|h|^{n+s}} U(x-h) w_{\d} (h) \right|^p dx \right)^{\frac{1}{p}} dh .
\end{split}
\]
Now, for all $h\in B(0,\d) \setminus \{0 \}$,
\[
 \begin{split}
 	\left( \int_{\O} \left|  \frac{\f(x) - \f(x-h)}{|h|^{n+s}} U(x-h) w_{\d} (h) \right|^p dx \right)^{\frac{1}{p}} \leq& [\f]_{0,1} \frac{w_{\d} (h)}{|h|^{n+s-1}} \left( \int_{\O} \left|U(x-h)\right|^p  \, dx \right)^{\frac{1}{p}} \\
 	 \leq&  [\f]_{0,1} \g (1-s) \rho_{\d} (h) \left\| U \right\|_{L^p (\O_{\d})} .
 \end{split}
\]
Therefore,
\[
 \left\| K^{s,\d}_{\f} (U) \right\|_{L^p (\O)} \leq (n+s-1) [\f]_{0,1} \left\| \rho_{\d} \right\|_{L^1 (\Rn)} \left\| U \right\|_{L^p (\O_{\d})} ,
\]
which completes the proof in this case.

The proof for $p=\infty$ is even simpler: for $U \in L^{\infty} (\O_\d)$ and a.e.\ $x \in \O$,
\[
 \left| K^{s,\d}_{\varphi}(U) (x) \right| \leq |c_{n,s}| [\f]_{0,1} \left\| U \right\|_{L^{\infty} (\O_{\d})} \int_{B(x,\d)} \frac{w_{\d} (x-y)}{|x-y|^{n+s-1}} \, dy = (n+s-1)  [\f]_{0,1} \left\| U \right\|_{L^{\infty} (\O_{\d})} \left\| \rho_{\d} \right\|_{L^1 (\Rn)} .
\]
Therefore,
\[
  \left\| K^{s,\d}_{\varphi}(U) \right\|_{L^{\infty} (\O)} \leq (n+s-1)  [\f]_{0,1} \left\| \rho_{\d} \right\|_{L^1 (\Rn)} \left\| U \right\|_{L^{\infty} (\O_{\d})} ,
\]
which concludes the proof of \ref{item:Kbounded}.

As a consequence of the previous argument, the integrals of Definition \ref{de:Ksd} are absolutely convergent for a.e.\ $x \in \O$.
For such $x$ and $U \in L^1 (\O_{\d}, \Rn)$ we have
\begin{align*}
 \tr K_{\f}^{s,\d} (U^T) (x) & = c_{n,s} \int_{B(x,\d)} \tr \left( \frac{\varphi(x)-\varphi(y)}{|x-y|^{n+s}} U(y) \otimes \frac{x-y}{|x-y|} w_\d(x-y) \right) dy \\
 & = c_{n,s} \int_{B(x,\d)} \frac{\varphi(x)-\varphi(y)}{|x-y|^{n+s}} U(y) \cdot \frac{x-y}{|x-y|} w_\d(x-y) \, dy = K_{\f}^{s,\d} (U) (x) ,
\end{align*}
which proves \ref{item:trK}.
\end{proof}

Now we introduce a product formula for the nonlocal gradient.

\begin{lem} \label{lem: Gradiente no local producto}
Let $1 \leq p < \infty$ and $\f \in C^{\infty} (\bar{\O}_{\d})$.

\begin{enumerate}[label=\alph*)]
\item\label{item:derivativeproductA} If $g \in H^{s,p,\d}(\O)$ then $\f g \in H^{s,p,\d} (\O)$ and
\begin{equation*}
 D_\d^s (\f g) = \f \, D_\d^s g + K_{\f}^{s,\d} (g) .
\end{equation*}
 
\item\label{item:derivativeproductB} If $g \in H^{s,p,\d}(\O, \Rn)$ then $\f g \in H^{s,p,\d} (\O, \Rn)$ and
\[
 D_\d^s(\f g) = \f \, D_\d^s g + K_{\f}^{s,\d} (g^T) .
\]
\end{enumerate}
\end{lem} 

\begin{proof}
We show \ref{item:derivativeproductA}; the proof of \ref{item:derivativeproductB} follows from an application of \ref{item:derivativeproductA} componentwise.
The function $\f$ has a $C^{\infty}_c (\Rn)$ extension, so we can suppose $\f \in C^{\infty}_c (\Rn)$.
First we assume $g \in C^{\infty}_c(\Rn)$.
For all $x \in \Rn$ we have
\begin{equation}\label{eq:gradientproduct1}
\begin{split}
	 D_\d^s (\f g) (x) &= c_{n,s} \int_{B(x,\d)} \frac{\f(x) g(x) - \f(x) g(y) + \f(x) g(y) - \f(y) g(y)}{|x-y|^{n+s}}\frac{x-y}{|x-y|} w_\d(x-y) \, dy \\
	 &= \f(x) D_\d^s g(x)+ K_{\f}^{s,\d} (g) (x) .
\end{split}
 \end{equation}

Now we consider $g \in H^{s, p,\d} (\O)$ and a sequence $\{ g_j \}_{j \in \N} \subset C^{\infty}_c (\Rn)$ converging to $g$ in $H^{s,p,\d}(\O)$.
Then $\{ \varphi g_j \}_{j \in \N}$ is a sequence in $C^{\infty}_c (\Rn)$ that clearly converges to $\varphi g$ in $L^p (\O_\d)$.
Let us check that $\{D_\d^s ( \varphi g_j) \}_{j \in \N}$ converges in $L^p (\O, \Rn)$.
Owing to \eqref{eq:gradientproduct1} we have
\[
D_\d^s ( \varphi g_j) = \varphi \, D_\d^s g_j + K^{s,\d}_{\varphi} (g_j) , \qquad j \in \N .
\]
Since $D_\d^s g_j \to D_\d^s g$ in $L^p (\O, \Rn)$ as $j\to \infty$, we also have that $\varphi \, D_\d^s g_j \to \varphi \, D_\d^s g$ in $L^p (\O, \Rn)$.
By Lemma \ref{Lema operador lineal delta}, as $g_j \to g$ in $L^p (\O_\d)$ as $j \to \infty$, we obtain that $K^{s,\d}_{\varphi} (g_j) \to K^{s,\d}_{\varphi} (g)$ in $L^p (\O, \Rn)$.
This shows the conclusion of \ref{item:derivativeproductA}.
\end{proof}


For $\phi \in H^{s,p,\d} (\O, \Rn)$ there is a natural relation between $D_\d^s \phi$ and $\diver_\d^s \phi$.
We also state the product formula for the divergence.
As usual, the divergence of a matrix is the vector whose components are the divergence of the rows.

\begin{lem}\label{le:trazadiv}
Let $1 \leq p < \infty$.
Let $\phi \in H^{s,p,\d} (\O, \Rn)$.
Then $\diver_\d^s \phi \in L^p (\O)$ and $\tr D_\d^s \phi = \diver_\d^s \phi$ a.e.\ in $\O$.
Moreover, for any $\f \in C^{\infty}(\bar{\O}_{\d})$, 
\begin{displaymath}
 \diver_\d^s (\f \phi)  = \f \diver_\d^s \phi + K_{\f}^{s,\d} (\phi)  ,
\end{displaymath}
and for any $\Phi \in H^{s,p,\d} (\O, \Rnn)$,
\[
 \diver_\d^s (\f \Phi)  = \f \diver_\d^s \Phi + K_{\f}^{s,\d} (\Phi) .
\]
\end{lem}
\begin{proof}
By extension, we can assume $\f \in C_c^{\infty}(\Rn)$.
Let $\{ \phi_j \}_{j \in \N} \subset C^{\infty}_c (\Rn, \Rn)$ be a sequence converging to $\phi$ in $ H^{s,p,\d} (\O, \Rn)$.
Having in mind that the integrals of Definition \ref{def: nonlocal gradient} are absolutely convergent, we obtain, for each $x \in \O$ and $j \in \N$,
\begin{align*}
 \tr D_\d^s \phi_j (x) &= c_{n,s} \tr\left(\int_{B(x,\d)}  \frac{\phi_j(x) - \phi_j(y)}{|x-y|^{n+s}} \otimes \frac{x-y}{|x-y|}w_\d(x-y)dy\right) \\
 & = c_{n,s} \int_{B(x,\d)} \tr \left(\frac{\phi_j(x) - \phi_j(y)}{|x-y|^{n+s}} \otimes \frac{x-y}{|x-y|} w_\d(x-y) \right) dy \\ 
 &= c_{n,s} \int_{B(x,\d)} \frac{\phi_j(x) - \phi_j(y)}{|x-y|^{n+s}} \cdot \frac{x-y}{|x-y|} w_\d(x-y) dy = \diver_\d^s \phi_j (x) .
\end{align*}
Since $D_\d^s \phi_j \to D_\d^s \phi$ in $L^p (\O, \Rnn)$, we obtain, succesively, that $\tr D_\d^s \phi_j \to \tr D_\d^s \phi$ in $L^p (\O)$, $\diver_\d^s \phi \in L^p (\O)$ and $\tr D_\d^s \phi = \diver_\d^s \phi$ in $L^p (\O)$.

Now we prove the product formula.
By the result above and applying Lemmas \ref{lem: Gradiente no local producto}\,\ref{item:derivativeproductB} and \ref{Lema operador lineal delta}\,\ref{item:trK}, we obtain
\[
 \diver_\d^s (\f \phi) = \tr D_\d^s (\f \phi)  = \tr \left( \f \, D_\d^s \phi + K_{\f}^{s,\d} (\phi^T) \right) = \f \div_\d^s \phi +  K_{\f}^{s,\d} (\phi) .
\]
The formula for $\Phi$ is immediate by applying componentwise the above formula.
\end{proof}

\section{From nonlocal to local}\label{se:from}

In this section we show that nonlocal gradients are in fact gradients of another function.
This idea was exploited in \cite{KrSc22} in the fractional context, whose fully analogue result in this nonlocal case (nonlocal gradients are gradients, and vice versa) is shown in \cite{CuKrSc22}; here we also see one implication: nonlocal gradients are gradients (Lemma \ref{le:DsdD}). 

We first recall from \cite[Lemma 4.2]{BeCuMo22} that $Q^s_{\d}$ is a radial $L^1 (\Rn)$ function with support in $B(0,\d)$ such that
\[
 \nabla Q_\d^s(x)=- (n-1+s) \frac{\rho_\delta{(x)}}{|x|}\frac{x}{|x|} , \qquad x \in \Rn \setminus \{ 0 \} .
\]

We will do convolutions with $Q^s_{\d}$ with functions defined in $\O_{\d}$.
Thus, with a small abuse of notation, given $u : \O_{\d} \to \R$ we will write $Q^s_{\d} * u$ as the function defined in $\O$ by
\[
 Q^s_{\d} * u (x) = \int_{B(x,\d)} Q^s_{\d} (x-y) u (y) \, dy ,
\]
whenever the integral is well defined.
In truth, $Q^s_{\d} * u$ is the restriction to $\O$ of the convolution $Q^s_{\d} * \bar{u}$, where $\bar{u} : \Rn \to \R$ is the extension by zero of $u$.
We show the boundedness of this operator.

\begin{lem}\label{le:conv}
For any $1 \leq p \leq \infty$, the map $u \mapsto Q^s_{\d} * u$ is linear and bounded from $L^p (\O_{\d})$ to $L^p (\O)$.
\end{lem}
\begin{proof}
Given any $u \in L^p (\O_{\d})$, we denote by $\bar{u}$ its extension to $\Rn$ by zero and use Young's inequality, so as to obtain that
\[
 \left\| Q^s_{\d} * u \right\|_{L^p (\O)} = \left\| Q^s_{\d} * \bar{u} \right\|_{L^p (\O)} \leq \left\| Q^s_{\d} * \bar{u} \right\|_{L^p (\Rn)} \leq \left\| Q^s_{\d} \right\|_{L^1 (\Rn)} \left\| \bar{u} \right\|_{L^p (\Rn)} = \left\| Q^s_{\d} \right\|_{L^1 (\Rn)} \left\| u \right\|_{L^p (\O_{\d})} .
\]
\end{proof}

The first part (the smooth case) of the following lemma was essentially proved in \cite{BeCuMo22}.
We provide the necessary details to arrive at the precise formulation we need, as well as to prove the $H^{s,p,\d}$ version.
As mentioned before, the result is akin to that simultaneously proved in \cite[Prop.\ 2.15]{CuKrSc22}.

\begin{lem}\label{le:DsdD}
\begin{enumerate}[label=\alph*)]
\item\label{item:DsdDa}
For all $u \in C^{\infty}_c (\Rn)$ we have that $Q^s_{\d} * u \in C^{\infty}_c (\Rn)$ and
\[
 D^s_{\d} u = D (Q^s_{\d} * u) = Q^s_{\d} * D u .
\]
Moreover, for all $\phi \in C^{\infty}_c (\Rn, \Rn)$ we have that
\[
 \diver^s_{\d} \phi = \div (Q^s_{\d} * \phi) = Q^s_{\d} * \diver \phi.
\]

\item\label{item:DsdDb}
Let $1 \leq p < \infty$.
Then the map $u \mapsto Q^s_{\d} * u$ is linear and bounded from $H^{s,p,\d} (\O)$ to $W^{1,p} (\O)$.
Moreover, for all $u \in H^{s,p,\d} (\O)$,
\[
 D^s_{\d} u = D (Q^s_{\d} * u) \quad \text{in } \O ,
\]
and for all $\phi \in H^{s,p,\d} (\O, \Rn)$,
\[
 \diver^s_{\d} \phi = \div (Q^s_{\d} * \phi) \quad \text{in } \O .
\]

\end{enumerate}
\end{lem}
\begin{proof}
We start with \ref{item:DsdDa}.
We have from \cite[Prop.\ 4.3]{BeCuMo22} that $D^s_{\d} u = Q^s_{\d} * D u$.
In fact, owing to a classic result in convolution (see, e.g., \cite[Props.\ 4.18 and 4.20]{Brezis}), $Q^s_{\d} * u \in C^{\infty}_c (\Rn)$ and $D(Q^s_{\d} * u) = Q^s_{\d} * D u$.

We now apply the equality $Q^s_{\d} * D \phi = D^s_{\d} \phi$ as well as Lemma \ref{le:trazadiv} to conclude that
\[
 Q^s_{\d} * \diver \phi = Q^s_{\d} * (\tr D \phi) = \tr (Q^s_{\d} * D \phi) = \tr D^s_{\d} \phi = \diver^s_{\d} \phi 
\]
and, additionally,
\[
 \tr (Q^s_{\d} * D \phi) = \tr D (Q^s_{\d} * \phi) = \div (Q^s_{\d} * \phi) .
\]

Now we show \ref{item:DsdDb}.
Let $\{ u_j \}_{j \in \N}$ be a sequence in $C^{\infty}_c (\Rn)$ converging to $u$ in $H^{s,p,\d} (\O)$.
Since $u_j \to u$ in $L^p (\O_{\d})$, by Lemma \ref{le:conv}, $Q^s_{\d} * u_j \to Q^s_{\d} * u$ in $L^p (\O)$.
On the other hand, $D^s_{\d} u_j \to D^s_{\d} u$ in $L^p (\O, \Rn)$ and, by \ref{item:DsdDa}, $D^s_{\d} u_j = D (Q^s_{\d} * u_j)$ for each $j \in \N$.
By the locality and closedness of the derivative operator, $Q^s_{\d} * u \in W^{1,p} (\O)$ and $D (Q^s_{\d} * u) = D^s_{\d} u$.
The same proof also shows the boundedness of the operator $u \mapsto Q^s_{\d} * u$ from $H^{s,p,\d} (\O)$ to $W^{1,p} (\O)$.
Moreover, by Lemma \ref{le:trazadiv},
\[
 \diver^s_{\d} \phi = \tr D^s_{\d} \phi = \tr D (Q^s_{\d} * \phi) = \div (Q^s_{\d} * \phi) .
\]
\end{proof}

\section{Nonlocal Piola's Identity, integration by parts of the determinant and weak continuity}\label{se:Piola}

In this section we adapt three classical results in the theory of polyconvexity to the nonlocal context: Piola's identity, integration by parts of the determinant and weak continuity of the minors.

Recall that the classical Piola identity asserts that, for smooth enough functions $u : \Omega \subset \Rn \to \Rn$ one has $\div \cof Du= 0$.
Of course, $\cof$ denotes the cofactor matrix, which satisfies $\cof A \, A^T = (\det A) \, I$ for every $A \in \Rnn$. 


\begin{prop}
\begin{enumerate}[label=\alph*)]
\item\label{item:PiolaA}
For all $u \in C^{\infty}_c (\Rn, \Rn)$,
\[
 \diver^s_{\d} \cof D^s_{\d} u = \div \cof D^s_{\d} u = \diver^s_{\d} \cof D u = 0 .
\]

\item\label{item:PiolaB}
For all $u \in H^{s,p,\d} (\O, \Rn)$ with $p \geq n-1$ and all $\f \in C^{\infty}_c (\O_{-\d})$,
\begin{equation*}
 \int_{\O} \cof D^s_{\d} u \, D^s_{\d} \f \, dx = 0 ,
\end{equation*}
while for all $\f \in C^{\infty}_c (\O)$,
\[
 \int_{\O} \cof D^s_{\d} u \, D \f \, dx = 0 .
\]
\end{enumerate}
\end{prop}
\begin{proof}
For \ref{item:PiolaA}, using Lemma \ref{le:DsdD} we find that $\diver^s_{\d} \cof D^s_{\d} u = Q^s_{\d} * \diver (\cof D^s_{\d} u)$, with
\[
 \diver (\cof D^s_{\d} u) = \diver (\cof (D (Q^s_{\d} * u))) = 0 ,
\]
since $\diver (\cof (D (Q^s_{\d} * u))) = 0$ by the classical Piola identity, having in mind that $Q^s_{\d} * u \in C^{\infty}_c (\Rn, \Rn)$.
Similarly, $\diver^s_{\d} \cof D u = Q^s_{\d} * \diver \cof D u = 0$.

For \ref{item:PiolaB} we have, by Lemma \ref{le:DsdD} and Piola's identity for Sobolev functions (see, e.g., \cite[Lemma 6.1]{Ball77} or \cite[Prop.\ 3.2.4.1]{GiMoSo98I}), that, for any $\f \in C^{\infty}_c (\O_{-\d})$,
\[
 \int_{\O} \cof D^s_{\d} u \, D^s_{\d} \f \, dx = \int_{\O} \cof D (Q^s_{\d} * u) \, D (Q^s_{\d} * \f) \, dx = 0 ,
\]
since $Q^s_{\d} * u \in W^{1, p} (\O, \Rn)$ and $Q^s_{\d} * \f \in C^{\infty}_c (\O)$.
Similarly, for $\f \in C^{\infty}_c (\O)$,
\[
 \int_{\O} \cof D^s_{\d} u \, D \f \, dx = \int_{\O} \cof D (Q^s_{\d} * u) \, D \f \, dx = 0 .
\]
\end{proof}



The integration by parts of the determinant is as follows.

\begin{prop}
Let $p \geq n-1$ and $q \geq \frac{n}{n-1}$.
Let $u \in H^{s,p,\d} (\O, \Rn)$ be with $\cof D^s_{\d} u \in L^q (\O, \Rnn)$.
Then $\det D^s_{\d} u \in L^{\frac{q(n-1)}{n}} (\O)$ and for all $\f \in C^{\infty}_c (\O)$,
\begin{equation*}
 \int_{\O} \det D^s_{\d} u \, \f \, dx = - \frac{1}{n} \int_{\O} Q_\d^s * u \cdot \cof D_\d^s u \, D \f \, dx.
\end{equation*}
\end{prop}
\begin{proof}
By Lemma \ref{le:DsdD} and the classical integration by parts of the determinant \cite[Th.\ 3.2]{MuQiYa94}, we have that $\det D (Q^s_{\d} * u) \in L^{\frac{q(n-1)}{n}} (\O)$ and
\[
\begin{split}
	  \int_{\O} \det D_\d^s u \, \f \, dx &= \int_{\O} \det D (Q_\d^s * u) \, \f \, dx \\
	  &= - \frac{1}{n} \int_{\O} Q_\d^s * u \cdot \cof D (Q_\d^s * u) \, D \f \, dx = - \frac{1}{n} \int_{\O} Q_\d^s * u \cdot \cof D_\d^s u \, D \f \, dx .
\end{split}
\]
\end{proof}

The weak continuity of minors is as follows.
As in the previous results, Lemma \ref{le:DsdD} reduces the analysis to the Sobolev case, and for this we use the sharpest result of the continuity of the determinant in this context, which is due to \cite{MuQiYa94} and \cite{HeMo10}.

\begin{teo}\label{th:wcontdet}
Let $p \geq n-1$ and $q \geq \frac{n}{n-1}$.
Let $\{ u_j \}_{j \in \N}$ be a sequence in $H^{s,p,\d} (\O, \Rn)$ such that $u_j \weakc u$ in $H^{s,p,\d} (\O, \Rn)$.
Then:
\begin{enumerate}[label=\alph*)]
\item\label{item:wcontdetA} If $k \in\N$ with $1 \leq k \leq n-2$ and $\mu$ is a minor of order $k$ then $\mu (D_\d^s u_j) \weakc \mu (D_\d^s u)$ in $L^{\frac{p}{k}} (\O)$ as $j \to \infty$.

\item\label{item:wcontcof} If $\cof D_\d^s u_j \weakc \vartheta$ in $L^1 (\O, \Rnn)$ for some $\vartheta \in L^q (\O, \Rnn)$ then $\vartheta = \cof D_\d^s u$.

\item\label{item:wcontdet}
If $\cof D_\d^s u_j \in L^q (\O, \Rnn)$ for each $j \in \N$, $\cof D_\d^s u_j \weakc \cof D_\d^s u$ in $L^1 (\O, \Rnn)$ and $\det D_\d^s u_j \weakc \t$ in $L^1 (\O)$ for some $\t \in L^1 (\O)$ then $\t = \det D^s_{\d} u$.

\end{enumerate}
\end{teo}
\begin{proof}
Define $v_j = Q^s_{\d} * u_j$ and $v = Q^s_{\d} * u$.
By Lemma \ref{le:DsdD}, we have $v_j \weakc v$ in $W^{1,p} (\O, \Rn)$, as well as $D v_j = D^s_{\d} u_j$ and $D v = D^s_{\d} u$.
Thus, we can apply a classical result on the convergence of minors (see, e.g., \cite[Th.\ 6.2]{Ball77} or \cite[Th.\ 8.20]{dacorogna}) to conclude that $\cof D v_j \weakc \cof D v$ in the sense of distributions and $\mu (D v_j) \weakc \mu (D v)$ in $L^{\frac{p}{k}} (\O)$ for every minor $\mu$ of order $k \in \{ 1, \ldots, n-2 \}$.
This proves \ref{item:wcontdetA}.
Under assumption \ref{item:wcontcof}, we have $\cof D v_j \weakc \vartheta$ in $L^1 (\O, \Rnn)$ so $\vartheta = \cof D v = \cof D_\d^s u$, which proves \ref{item:wcontcof}.

Suppose assumption \ref{item:wcontdet}.
As $\cof D v_j \in L^q (\O, \Rnn)$ for each $j \in \N$, we conclude that the \emph{divergence identities} of \cite[Th.\ 3.2]{MuQiYa94} hold for $v_j$.
This, in turn, is equivalent to the fact that the surface energy $\bar{\mathcal{E}}$ of $v_j$ defined in \cite[Def.\ 1]{HeMo10} is zero (see \cite[Prop.\ 3]{HeMo10}).
As observed in \cite[Sect.\ 2.5]{BaHeMo17}, the surface energy $\mathcal{E}$ of $v_j$ defined in \cite[Def.\ 2]{HeMo10} is also zero.
Thus, we can apply \cite[Th.\ 3]{HeMo10} and conclude that $\t = \det D v = \det D^s_{\d} u$.
\end{proof}

\section[Existence of minimizers]{Existence of minimizers and Euler--Lagrange equations}\label{se:existence}

In this section we prove the existence of minimizers in $H^{s,p,\d}$ of functionals of the form
\begin{equation}\label{eq:I}
 I (u) := \int_{\O} W (x, u(x), D^s_\d u (x)) \, dx .
\end{equation}
under natural coercivity and polyconvexity assumptions.
We also derive the associated Euler--Lagrange equation, which is a partial nonlocal-differential system of equations.

We recall the concept of polyconvexity (see, e.g, \cite{Ball77,dacorogna}).
Let $\tau$ be the number of submatrices of an $n \times n$ matrix.
We fix a function $\mu : \Rnn \to \R^{\tau}$ such that $\mu (F)$ is the collection of all minors of an $F \in \Rnn$ in a given order.
A function $W_0 : \Rnn \to \R \cup \{\infty\}$ is polyconvex if there exists a convex $\Phi : \R^{\tau} \to \R \cup \{\infty\}$ such that $W_0 (F) = \Phi (\mu (F))$ for all $F \in \Rnn$.

The existence theorem is as follows.
Its proof relies on a standard argument in the Calculus of variations, once we have the continuity (with respect to the weak convergence) of the minors given by Theorem \ref{th:wcontdet}.

\begin{teo}\label{th:existence}
Let $p \geq n-1$ satisfy $p>1$.
Let $q \geq \frac{n}{n-1}$.
Let $W : \Rn \times \Rn \times \Rnn \to \R \cup \{ \infty \}$ satisfy the following conditions:
\begin{enumerate}[label=\alph*)]
\item $W$ is $\mc{L}^n \times \mc{B}^n \times \mc{B}^{n\times n}$-measurable, where $\mc{L}^n$ denotes the Lebesgue sigma-algebra in $\Rn$, whereas $\mc{B}^n$ and $\mc{B}^{n\times n}$ denote the Borel sigma-algebras in $\Rn$ and $\Rnn$, respectively.

\item $W (x, \cdot, \cdot)$ is lower semicontinuous for a.e.\ $x \in \Rn$.

\item For a.e.\ $x \in \Rn$ and every $y \in \Rn$, the function $W (x, y, \cdot)$ is polyconvex.

\item\label{item:Ecoerc}
There exist a constant $c>0$, an $a \in L^1 (\O)$ and a Borel function $h : [0, \infty) \to [0, \infty)$ such that
\[
 \lim_{t \to \infty} \frac{h (t)}{t} = \infty
\]
and
\[
 W (x, y, F) \geq a (x) + c \left| F \right|^p + c \left| \cof F \right|^q + h (\left| \det F \right|)
\]
for a.e.\ $x \in \O$, all $y \in \Rn$ and all $F \in \Rnn$.
\end{enumerate}
Let $u_0 \in H^{s,p,\d} (\O, \Rn)$.
Define $I$ as in \eqref{eq:I}, and assume that $I$ is not identically infinity in $H^{s,p,\d}_{u_0} (\O_{-\d}, \Rn)$.
Then there exists a minimizer of $I$ in $H^{s,p,\d}_{u_0} (\O_{-\d},\Rn)$.
\end{teo}
\begin{proof}
Assumption \ref{item:Ecoerc} shows that the functional $I$ is bounded below by $\int_{\O} a$.
As $I$ is not identically infinity in $H^{s,p,\d}_{u_0} (\O, \Rn)$, there exists a minimizing sequence $\{ u_j \}_{j \in \N}$ of $I$ in $H^{s,p,\d}_{u_0} (\O, \Rn)$.
Assumption \ref{item:Ecoerc} implies that $\{ D^s_\d u_j \}_{j \in \N}$ is bounded in $L^p (\O, \Rnn)$.
Applying Theorem \ref{th:PoincareSobolev delta} to $u_j - u_0$, we obtain that $\{ u_j \}_{j \in \N}$ is bounded in $L^p (\O, \Rn)$ and, consequently, also in $H^{s,p,\d} (\O, \Rn)$.
Using that $H^{s,p,\d} (\O, \Rn)$ is reflexive (Proposition \ref{prop: espacio separable y reflexivo}) and Theorem \ref{Hspdelta embedding theorem}, there exists $u \in H^{s,p,\d}_{u_0} (\O_{-\d}, \Rn)$ such that for a subsequence (not relabelled),
\begin{equation*}
 u_j \weakc u \text{ in } H^{s,p,\d} (\O, \Rn) \quad \text{and} \quad u_j \to u \text{ in } L^p (\O, \Rn) .
\end{equation*}
By Theorem \ref{th:wcontdet}, for any minor $\mu$ of order $k \leq n-2$, we have that
\begin{equation*}
 \mu (D^s_\d u_j) \weakc \mu (D^s_\d u) \text{ in } L^{\frac{p}{k}} (\O) .
\end{equation*}
Moreover, assumption \ref{item:Ecoerc} shows that $\{ \cof D^s_\d u_j \}_{j \in \N}$ is bounded in $L^q (\O, \Rnn)$, and, as $q >1$, we can extract a weakly convergent subsequence, which by Theorem \ref{th:wcontdet} tends to $\cof D^s_\d u$.
Again by \ref{item:Ecoerc}, together with de la Vall\'ee Poussin's criterion, we have that $\{ \det D^s_\d u_j \}_{j \in \N}$ is equiintegrable, so, for a subsequence, it converges weakly in $L^1 (\O)$ to a function that, by Theorem \ref{th:wcontdet} has to be  $\det D^s_\d u$.

The above convergences imply, thanks to a standard lower semicontinuity result for polyconvex functionals (see, e.g., \cite[Th.\ 5.4]{BallCurrieOlver} or \cite[Th.\ 7.5]{FoLe07}), that
\[
 I(u) \leq \liminf_{j \to \infty} I(u_j) .
\]
Therefore, $u$ is a minimizer of $I$ in $H^{s,p,\d}_{u_0} (\O_{-\d}, \Rn)$.
\end{proof}

It is well known that Sobolev functions cannot have discontinuities along $(n-1)$-dimensional surfaces (see, e.g., \cite[Thms.\ 4.7.4, 4.8.1 and 5.6.3]{EvGa92}).
On the other hand, depending on the exponent of integrability of $Du$ and $\cof Du$, a Sobolev function $u$ can or cannot have discontinuities of cavitation type.
In fact, under the assumptions of \cite{MuQiYa94}, the deformation cannot exhibit cavitation.
But, as a consequence of the analysis of \cite[Sect.\ 2.1]{BeCuMC} and \cite[Prop.\ 3.5]{BeCuMo22} (see also \cite[Prop.\ 1.8]{CoSt23}), functions in $H^{s,p,\d} (\O)$ can exhibit fracture in the range of exponents $s p < 1$, and cavitation in the range of exponents $s p < n$.
Both ranges are compatible with the assumptions of Theorem \ref{th:existence}.

This is in contrast with the Sobolev or $SBV$ (special bounded variation) case.
Indeed, the classical space for fracture is $SBV$ and the energy to minimize requires a contribution of the surface energy (see, e.g., \cite{AmFuPa00}).
On the other hand, it was shown in \cite{HeMo11,HeMo12} that the divergence identities of \cite{MuQiYa94} correspond to the absence of cavitation, while these divergence identities were essential to prove Theorem \ref{th:wcontdet}.
The explanation of this apparent paradox is that the key of the proof of Theorem \ref{th:wcontdet} is that $Q^s_{\d} * u$ does not exhibit cavitation or fracture, regardless of whether $u$ does.
In fact, in many proofs of this article we invoke known results to the regularized version $Q^s_{\d} * u$ of $u$.

To sum up, Theorem \ref{th:existence}, which determines the existence of minimizers of a nonlocal hyperelastic energy, is compatible with functions exhibiting fracture or cavitation, in opposition to the case of classical elasticity (see, e.g., \cite{Ball77,Sverak88,MuQiYa94,Ball01,Ball02,BaHeMo17}).
Thanks to the sharpness of exponents $p,q$ of Theorem \ref{th:existence}, the range of exponents $p,q,s$ for which fracture or cavitation are compatible with the existence of minimizers is wider than that of the fractional case \cite{BeCuMC}.

To finish the article, we show the equilibrium (Euler--Lagrange) equations that minimizers of functional \eqref{eq:I} satisfy.
The notation for partial derivatives is as follows: $D_y W (x, \cdot, F)$ is the derivative of $W (x, \cdot, F)$, and $D_F W (x, y, \cdot)$ is the derivative of $W (x, y, \cdot)$.

\begin{teo}
Let $1 < p < \infty$.
Let $u_0 \in H^{s,p,\d} (\O)$.
Let $W : \O \times \Rn \times \Rnn \to \R$ satisfy the following conditions:
\begin{enumerate}[label=\alph*)]
\item $W (\cdot , y, F)$ is $\mc{L}^n$-measurable for each $y \in \Rn$ and $F \in \Rnn$.

\item $W (x, \cdot, \cdot)$ is of class $C^1$ for a.e.\ $x \in \O$.

\item
There exist $c>0$ and $a \in L^1 (\O)$ such that some of the following inequalities holds for a.e.\ $x \in \O$, all $y \in \Rn$ and all $F \in \Rnn$:
\begin{enumerate}[label=c\arabic*)]
\setcounter{enumii}{-1}
\item\label{item:c0}
\begin{equation*}
\left| W (x, y, F) \right| + \left| D_y W (x, y, F) \right| + \left| D_F W (x, y, F) \right| \leq a (x) + c \left( \left| y \right|^p + \left| F \right|^p \right) ,
\end{equation*}

\item\label{item:c1} $s p < n$, $u_0 \in L^r (\O_{\d}, \Rn)$ with $1 \leq r \leq p_s^*$ and
\[
 \left| W (x, y, F) \right| + \left| D_y W (x, y, F) \right| + \left| D_F W (x, y, F) \right| \leq a (x) + c \left( \left| y \right|^r + \left| F \right|^p \right) .
\]

\item\label{item:c2} $s p = n$, 
\[
 \int_{\O_{\d}} \exp \left( \left|u_0 (x)\right|^r \right) dx < \infty 
\]
for some $1 \leq r < p'$ and 
\[
 \left| W (x, y, F) \right| + \left| D_y W (x, y, F) \right| + \left| D_F W (x, y, F) \right| \leq a (x) + \exp (|y|^r) + c \left| F \right|^p .
\]

\item\label{item:c3} $s p > n$, $u_0 \in L^{\infty} (\O_{\d}, \Rn)$ and there exists a function $f : [0, \infty) \to [0, \infty)$ sending bounded sets in bounded sets such that
\[
 \left| W (x, y, F) \right| + \left| D_y W (x, y, F) \right| + \left| D_F W (x, y, F) \right| \leq a (x) + f ( |y| ) + c \left| F \right|^p .
\]
\end{enumerate}
\end{enumerate}
Define $I$ as in \eqref{eq:I}.
Let $u$ be a minimizer of $I$ in $H_{u_0}^{s,p,\d} (\O_{-\d})$.
Then, for every $\f \in C^{\infty}_c (\O_{-\d})$,
\begin{equation*}
 \int_{\O} \left[ D_y W (x, u(x), D^s_\d u (x)) \, \f (x) + D_F W (x, u(x), D^s_\d u (x)) \cdot D^s_\d \f (x) \right] dx = 0 .
\end{equation*}
If, in addition, $D_z W (\cdot, u(\cdot), D^s_\d u (\cdot)) \in C^1 (\overline{\O_{-\d}}, \Rn)$ then
\begin{equation*}
 D_y W (x, u(x), D^s_\d u (x)) = \diver^s_{\d} D_F W (x, u(x), D^s_\d u (x))
\end{equation*}
for a.e.\ $x \in \O_{-\d}$.
\end{teo}
\begin{proof}
The proof for the scalar case and under assumption \ref{item:c0} was shown in \cite[Th.\ 8.2]{BeCuMo22}.
The proof for the vectorial case requires no essential modification, so we just mention the small adaptations for cases \ref{item:c1}--\ref{item:c3}.

Thanks to Theorem \ref{th:PoincareSobolev delta}, $u \in L^r (\O_{\d}, \Rn)$ under assumption \ref{item:c1},
\[
 \int_{\O_{\d}} \exp \left( \left|u (x)\right|^r \right) dx < \infty 
\]
under assumption \ref{item:c2}, and $u \in L^{\infty} (\O_{\d}, \Rn)$ under assumption \ref{item:c3}.
Actually, case \ref{item:c2} requires the Trudinger inequality of \cite[Th.\ 6.4]{BeCuMo22}.
The corresponding bounds of \ref{item:c1}--\ref{item:c3} allow us to apply dominated convergence to differentiate under the integral sign.
This concludes the proof.
\end{proof}

\section*{Acknowledgements} 

This work has been supported by the Spanish {\it Agencia Estatal de Investigaci\'on, Ministerio de Ciencia e Innovaci\'on} through projects PID2020-116207GB-I00 (J.C.B. and J.C.), and PID2021-124195NB-C32 (C.M.-C.), {\it Junta de Comunidades de Castilla-La Mancha} through grant SBPLY/19/180501/000110 and European Regional Development Fund 2018/11744 (J.C.B. and J.C.).
J.C.\ is also supported by Fundaci\'on Ram\'on Areces.

\bibliography{bibliography}{}
\bibliographystyle{siam} 

\end{document}